\documentclass[10pt]{article}

\usepackage[margin=2cm]{geometry}
\usepackage{xcolor}

\usepackage{soulutf8}

\usepackage[T2A]{fontenc}
\usepackage[russian]{babel}

\usepackage{intcalc, calc}
\usepackage[nomessages]{fp}
\usepackage{graphicx}
\usepackage{wrapfig}
\usepackage{float}

\usepackage{tikz}
\usepackage{color}
\usepackage{framed}
\usepackage{ragged2e}
\usepackage{amsmath, amssymb, amsfonts, amsthm}
\usepackage{wasysym}

\usepackage{calc}
\usepackage{mathtools,mathcomp}
\usepackage{pifont}
\usepackage{enumitem}
\usepackage{hyperref}

\usepackage{ifthen}

\def\divdots{\rlap{\raisebox{-1pt}{.}}{\rlap{\raisebox{2pt}{.}}\raisebox{5pt}{.}}}

\def\ndivby{\mathrel{
    \divdots
    \kern-0.35em\raise0.22ex\hbox{/}
}}

\renewcommand{\leq}{\leqslant}
\renewcommand{\geq}{\geqslant}

\def\fs{\kern 0.5em}
\newcounter{prcnt}
\newcounter{pucnt}
\newcommand{\prmain}[1]{ 
    \medskip%
    \setcounter{pucnt}{0}%
    \stepcounter{prcnt}%
    \noindent\textbf{%
    \theprcnt%
    \ifthenelse{\equal{#1}{1}}{*}{}%
    .}\fs%
}

\newcommand{\pumain}[1]{
    \stepcounter{pucnt}{%
    \noindent\bf(\alph{pucnt}%
    \ifthenelse{\equal{#1}{1}}{*}{}%
    )}\fs%
}

\newtheorem{theorem}{Теорема}
\newtheorem{proposition}{Утверждение}
\newtheorem{lemma}{Лемма}

\newtheorem*{theorem*}{Теорема}

\theoremstyle{definition}

\DeclareMathOperator{\rk}{rk}
\DeclareMathOperator{\cutrk}{cutrk}

\title{О стабильности взвешенного степенного перечислителя остовных деревьев}
\author{Павел Прозоров$^{\mathrm{a}}$ и Данила Черкашин$^\mathrm{a,b}$\\
{\small ~a. Санкт-Петербургский Государственный Университет, Россия}\\
{\small ~b. Институт Математики и Информатики Болгарской Академии Наук, София, Болгария}}
\begin{document}

\maketitle

\begin{abstract}
    Paper~\cite{cherkashin2023stability} shows that the (vertex) spanning tree degree enumerator polynomial of a connected graph $G$ is a real stable polynomial (id est is non-zero if all variables have positive imaginary parts) if and only if $G$ is distance-hereditary.
    In this note we generalize the result on weighted graphs. 
    
    This generalization allows us to define the class of weighted distance-hereditary graphs.
     
    %We show that the spanning tree degree enumerator polynomial of a connected graph G is a real stable polynomial if and only if G is distance-hereditary.
\end{abstract}

\begin{abstract}
    В статье~\cite{cherkashin2023stability} показано, что степенной (вершинный) перечислитель остовных деревьев связного графа $G$ является вещественно-стабильным многочленом (то есть не обнуляется при подстановке переменных с положительными мнимыми частями) тогда и только тогда, когда $G$ принадлежит классу дистанционно-наследуемых графов.
    В данной заметке мы обобщаем данный результат на взвешенные графы. 
    
    Полученное обобщение позволяет определить класс взвешенных дистанционно-наследуемых графов. 
     
    %We show that the spanning tree degree enumerator polynomial of a connected graph G is a real stable polynomial if and only if G is distance-hereditary.
\end{abstract}

\section{Введение}

Определим верхнюю комплексную полуплоскость
\[
\mathbb{H}:=\{ z \in\mathbb{C}|\Im(z)>0\}.
\]
Полином $P(x_1,x_2, \dots,x_n)$ с вещественными коэффициентами называется \textit{вещественно ста\-биль\-ным}, если $P(z_1,z_2, \dots, z_n)\neq 0$  для любых $z_1,z_2,\dots, z_n \in \mathbb{H}$.
Ясно, что ненулевой многочлен вида 
\[
a_1x_1+a_2x_2+\dots+a_nx_n, \quad \quad a_i \in \mathbb{R}_+
\]
является вещественно стабильным, и очевидно, что произведение двух вещественно стабильных многочленов есть вещественно стабильный многочлен.
Следующее утверждение является широко известным, например, см.~\cite{wagner2011multivariate}.
\begin{proposition}
\label{pr:basic}
Пусть $P(x_1,x_2, \dots, x_n)$ --- вещественно стабильный многочлен. 
Тогда следующие многочлены также являются стабильными или тождественным нулем
\begin{itemize}
\item[(i)] $x_1^{d_1}P\left(-\frac{1}{x_1}, x_2, \dots, x_n\right)$, где $d_1$ --- степень $P$ по переменной $x_1$;
\item[(ii)] $\frac{\partial P}{\partial x_1}(x_1,x_2, \dots, x_n)$;
\item[(iii)] $Q(x_1,x_2, \dots, x_{n-1}) :=  P(x_1,x_2, \dots, x_{n-1},a)$ при любом вещественном $a$.
\end{itemize}
\end{proposition}

Больше о стабильных многочленов и их применениях можно узнать в обзорах~\cite{wagner2011multivariate,Csikvari2022ASS} (также см. комментарии авторов в разделах 1 и 4 статьи~\cite{cherkashin2023stability}).

Пусть $G=(V, E)$ --- конечный простой связный неориентированный граф, и пусть $|V|=n, |E|=k$.
Пусть $N_G(v)=\{u\in V\colon vu\in E\}$ --- окрестность вершины $v$, а $\deg_G(v)$ --- степень вершины $v$ в графе $G$.
Для подмножества вершин $U \subset V$ определим \textit{индуцированный подграф} $G[U]$ как граф, вершинами которого являются элементы $U$, а рёбра 
суть те рёбра графа $G$, оба конца которых
лежат в $U$.
Обозначим через $S(G)$ множество всех остовных деревьев графа $G$. 
Обозначим полный граф на $n$ вершинах и полный двудольный граф с долями размера $n$ и $m$ через $K_n$ и $K_{n,m}$, соответственно. 
\textit{Двусвязным} называется связный граф, остающийся связным при удалении любой вершины (и всех инцидентных ей рёбер).

Пронумеруем рёбра $G$ числами от 1 до $k$, и каждому ребру
$i=1,\ldots,k$ сопоставим переменную $x_i$.
Определим \textit{реберный остовный многочлен} графа $G$
\[
Q_G(x_1,x_2, \dots, x_k )=\sum_{T\in S(G) }\prod_{j \in E(T)}x_j.
\]
Как известно~\cite{choe2004homogeneous}, многочлен $Q_G$ является вещественно стабильным для любого конечного связного простого графа $G$, содержащего не меньше
двух вершин.
Можно пронумеровать вершины от 1 до $n$, также сопоставить
им переменные $x_1,\ldots,x_n$ и оп\-ре\-де\-лить \textit{вершинный остовный многочлен} 
\[
P_G(x_1, x_2, \dots, x_n)=\sum_{T \in S(G)} \prod_{v \in V}x_v^{\deg_T(v)-1},
\]

Этот многочлен не всегда является вещественно стабильным. Назовем \textit{стабильными} графы, для которых многочлен $P_G$ является вещественно стабильным. Назовем граф \textit{дистанционно-наследуемым}, если для любого его связного индуцированного подграфа расстояние между любыми двумя  вершинами в подграфе равно расстоянию между ними в исходном графе (более подробно о дистанционно-наследуемых графах см. в разделе~\ref{dh}).
В статье~\cite{cherkashin2023stability} показано, что эти классы совпадают.

\begin{theorem}[Петров -- Прозоров -- Черкашин~\cite{cherkashin2023stability}]
\label{th:main}
Стабильными графами являются в точности те графы, которые являются дистанционно-наследуемыми.
\end{theorem}

В статье~\cite{cherkashin2023stability} приведены комбинаторная мотивация и приложения теоремы~\ref{th:main}, а также сформулированы связанные открытые вопросы. Решению одного из них --- получению классификации стабильных взвешенных графов --- и посвящена эта статья.

Определим \textit{взвешенный вершинный остовный многочлен}  графа $G$ c функцией весов $w$ из $E(G)$ в $\mathbb{R}$
\[
P_{G,w}(x_1, x_2, \dots, x_n)= \sum_{T \in S(G)} \prod_{e \in T}w(e) \prod_{v \in V} x_v^{\deg_T(v)-1} .
\]

Назовем взвешенный граф $(G,w)$ \textit{стабильным}, если многочлен $P_{G,w}$ стабильный. 
Не умаляя общности, можно считать, что $w$ не обращается в $0$ (такие рёбра можно удалить).

В статье~\cite{cherkashin2023stability} показано, что можно считать $w$ положительной функцией; мы повторим это рассуждение для полноты изложения. Если взвешенный граф не двусвязен, то можно менять знак $w$ в каждой компоненте двусвязности: по лемме~\ref{lm:gluing} нули многочлена $P_{G,w}$ зависят только от его редукций на компоненты двусвязности. 

\begin{lemma}
    Если взвешенный двусвязный граф $(G,w)$ стабилен, то $w$ имеет фиксированный знак на всех ребрах $G$.  
\end{lemma}
\begin{proof}
Предположим противное. Тогда в графе есть два ребра разного знака. Тогда в силу связности есть два ребра, имеющих общий конец, разного знака. Пусть это вершина $v$, и ребра --- $vu_1$  и $vu_2$, и $w(vu_1)>0>w(vu_2)$. Заметим, что при подстановке $x_v=0$ мы получим 
\[
P_{G \setminus v} (x_1, \dots, x_n) \cdot \left (\sum_{t \in N(v)} w(v,t)x_t\right).
\]
Тогда оба этих многочлена не являются тождественным нулем (т.к. $G \setminus v$ связный), а значит они оба стабильны, но второй многочлен очевидно не является стабильным, потому что в него можно подставить все переменные, кроме $x_{u_1}, x_{u_2}$ нулями, и он должен остаться стабильным, но при подстановке $x_{u_1}=-iw(vu_2), x_{u_2}=iw(vu_1)$ он обращается в ноль. Противоречие.
\end{proof}
Значит внутри каждой компоненты двусвязности знак одинаковый, и в каждой компоненте мы можем поменять знак, и это не повлияет на стабильность всего графа. Значит достаточно классифицировать графы, в которых все веса положительны (произвольный взвешенный граф стабилен тогда и только тогда, когда получается сменами знаков в компонентах двусвязности из взвешенного стабильного графа с положительными весами).

Далее везде будем считать все графы взвешенными, а все веса положительными, если не оговорено противное.
Следующая теорема дает явную (проверяемую за полиномиальное время) характеризацию взвешенных стабильных графов.

\begin{theorem}
\label{mainweighted}
Взвешенно стабильными графами с положительной $w$ являются в точности графы, получаемые из графа с одной вершиной путем копирований вершин, сохраняющих веса (с добавлением ребра произвольного положительного веса или без него), сочленений двух взвешенно-стабильных графов по вершине и умножений всех ребер из одной вершины на положительное число. 
\end{theorem}

%\textcolor{red}{Кажется, доказательство второй части вообще нигде не написано, но оно просто аналогично же домножением на линейную скобочку (см. замечание про знаки снова).}

Отметим, что класс взвешенно-стабильных графов не сводится к классу стабильных графов с помощью операций домножения. 
Иллюстрирующим примером является полный граф на четырех вершинах, пять ребер которого единичного веса, а шестое имеет вес 2.

\paragraph{Структура статьи.}  В разделе~\ref{dh} приводятся эквивалентные определения дистанционно-наследуемых графов, часть из которых мы используем.
В разделе~\ref{proof} доказывается основная теорема. В разделе~\ref{diss} обсуждается связь полученного многочленного определения взвешенного дистанционно-наследуемого графа и других определений.

\section{Дистанционно-наследуемые графы}
\label{dh}

Напомним, что дистанционно-наследуемый граф --- это граф, у которого любой связный индуцированный подграф сохраняет расстояние между любыми двумя вершинами. 
Более подробно о дистанционно-наследуемых графах можно прочесть, например, в статье~\cite{bandelt1986distance} и книге~\cite{brandstadt1999graph}.
В частности, в этой книге даны следующие эквивалентные определения класса дистанционно-наследуемых графов:

\begin{itemize}
    \item[(i)] Это графы, в которых любой порождённый путь является кратчайшим.
    \item[(ii)] Это графы, в которых любой цикл длины по меньшей мере пять имеет две или более диагоналей и в которых любой цикл длины в точности пять имеет по меньшей мере одну пару пересекающихся диагоналей.
    \item[(iii)] Это графы, в которых любой цикл длины пять и более имеет по меньшей мере одну пару пересекающихся диагоналей.
    \item[(iv)] Это графы, в которых для любых четырёх вершин $u$, $v$, $w$ и $x$ по меньшей мере две из трёх сумм расстояний $d(u,v)+d(w,x)$, $d(u,w)+d(v,x)$ и $d(u,x)+d(v,w)$ равны.
    \item[(v)] Это графы, в которых отсутствуют следующие индуцированные подграфы: цикл длины пять и более, изумруд, дом или домино (см. рис.~\ref{fig:forbidden}).
\begin{figure}[H]
    \centering
        \hfill  
   \begin{tikzpicture}[scale=1.8]
    
    \clip(1.5,0.8) rectangle (3.5,2.8);

    \coordinate(a1) at (2,1);
    \coordinate(a2) at (3,1);
    \coordinate(a3) at (3.3,1.9);
    \coordinate(a4) at (2.5,2.6);
    \coordinate(a5) at (1.7,1.9);

    \draw [blue,ultra thick] (a1)--(a2)--(a3)--(a4)--(a5);
    \draw [blue,ultra thick,dashed] (a5)--(a1);

    \fill (a1) circle (1.2pt);
    \fill (a2) circle (1.2pt);
    \fill (a3) circle (1.2pt);
    \fill (a4) circle (1.2pt);
    \fill (a5) circle (1.2pt);

\end{tikzpicture}
     \hfill
   \begin{tikzpicture}[scale=1.8,rotate=180]
    
    \clip(1.5,0.8) rectangle (3.5,2.8);

    \coordinate(a1) at (2,1);
    \coordinate(a2) at (3,1);
    \coordinate(a3) at (3.3,1.4);
    \coordinate(a4) at (2.5,2.6);
    \coordinate(a5) at (1.7,1.4);

    \draw [blue,ultra thick] (a1)--(a2)--(a3)--(a4)--(a5)--(a1);
    \draw [blue,ultra thick] (a1)--(a4)--(a2);
    
    \fill (a1) circle (1.2pt) node[right]{4};
    \fill (a2) circle (1.2pt) node[left]{3};
    \fill (a3) circle (1.2pt) node[left]{2};
    \fill (a4) circle (1.2pt) node[below]{1};
    \fill (a5) circle (1.2pt) node[right]{5};

\end{tikzpicture}
    \hfill 
    \begin{tikzpicture}[scale=1.8]
    
    \clip(1.5,0.8) rectangle (3.5,2.8);

    \coordinate(a1) at (2,1);
    \coordinate(a2) at (3,1);
    \coordinate(a3) at (3.3,1.9);
    \coordinate(a4) at (2.5,2.6);
    \coordinate(a5) at (1.7,1.9);

    \draw [blue,ultra thick] (a1)--(a2)--(a3)--(a4)--(a5)--(a1);
    \draw [blue,ultra thick] (a3)--(a5);

    \fill (a1) circle (1.2pt) node[below]{1};
    \fill (a2) circle (1.2pt) node[below]{5};
    \fill (a3) circle (1.2pt) node[below right]{4};
    \fill (a4) circle (1.2pt) node[below]{3};
    \fill (a5) circle (1.2pt) node[below left]{2};

\end{tikzpicture}
    \hfill 
    \begin{tikzpicture}[scale=1.8]
    
    \clip(1.5,0.8) rectangle (3.5,2.8);

    \coordinate(a1) at (1.7,1);
    \coordinate(a2) at (2.5,1);
    \coordinate(a3) at (3.3,1);
    \coordinate(a4) at (1.7,2.6);
    \coordinate(a5) at (2.5,2.6);
    \coordinate(a6) at (3.3,2.6);

    \draw [blue,ultra thick] (a1)--(a2)--(a3)--(a6)--(a5)--(a4)--(a1);
    \draw [blue,ultra thick] (a2)--(a5);

    \fill (a1) circle (1.2pt) node[below left]{6};
    \fill (a2) circle (1.2pt) node[below left]{1};
    \fill (a3) circle (1.2pt) node[below left]{2};
    \fill (a4) circle (1.2pt) node[below left]{5};
    \fill (a5) circle (1.2pt) node[below left]{4};
    \fill (a6) circle (1.2pt) node[below left]{3};

\end{tikzpicture}
    \caption{Запрещенные индуцированные подграфы слева направо: длинный цикл, изумруд, дом и домино}
    \label{fig:forbidden}
\end{figure}
    \item[(vi)] Это графы, которые могут быть созданы из одной вершины с помощью последовательности из следующих трёх операций:
    \begin{itemize}
        \item Добавление новой висячей вершины, соединённой одним ребром с существующей вершиной графа.
        \item Замена любой вершины графа парой вершин, каждая из которых имеет тех же соседей, что и удалённая вершина. 
        \item Замена любой вершины графа парой вершин, каждая из которых имеет тех же соседей, что и удалённая вершина, включая другую вершину из пары. 

    \end{itemize}
    
\end{itemize}

Отметим, что доказательство теоремы~\ref{th:main} основывалось на эквивалентности определений~(v) и~(vi).
Доказательство теоремы~\ref{mainweighted} также использует определения~(v) и~(vi), а также идею, стоящую за определением~(iv).

Еще одно алгебраическое определение появилось позднее в статье~\cite{oum2005rank}. Мы подробно обсудим его в разделе~\ref{diss}.

\section{Доказательство теоремы~\ref{mainweighted}}
\label{proof}

Напомним, что по умолчанию все графы являются взвешенными с положительными весами.

\subsection{Любой построенный граф является взвешенно-стабильным}

Следующие леммы являются взвешенными обобщениями лемм из статьи~\cite{cherkashin2023stability}.

\begin{lemma}[о копировании] \label{lm:copy}
    Пусть дан взвешенный граф $(G,w)$ на $n$ вершинах. Рассмотрим взвешенный граф $(G_1,w_1)$, получаемый из графа $(G,w)$ добавлением вершины $v_{n+1}$, соединением ее с вершинами из $N(v_n)$ ребрами, равными по весу ребрам из $v_n$, и с вершиной $v_n$ ребром веса $p \geq 0$ (отметим, что случай $p=0$ соответствует непроведению этого ребра).
    Тогда 
    \[
    P_{(G_1,w_1)}(x_1, \dots, x_n, x_{n+1})=P_{G,w}(x_1, \dots, x_n+x_{n+1}) \left( \sum_{t \in N(v_n)}w(tv_n)x_t+ p(x_n+x_{n+1}) \right).
    \]
\end{lemma}
\begin{proof}
Рассуждения аналогичны случаю невзвешенных графов.
Действительно, давайте заметим, что любое дерево в графе $G$ устроено следующим образом --- на всех вершинах кроме $v_n$ берется некоторый лес, такой, что в каждой компоненте есть хотя бы одна вершина из $N_G(v_n)$, и потом вершина $v_n$ соединяется с ровно одной вершиной из каждой компоненты.
Обозначим за $L$ множество всех таких лесов, за $t(K)$ --- число компонент связности в лесу $K$, 
и назовем $A_1, A_2, \dots, A_t$ пересечения множества $N_G(v)$ с компонентами связности леса $K$.
Пусть 
\[
W(K)= \prod_{uv \in K}w(uv)x_ux_v.
\]
Тогда из рассуждения выше 
\[
P_{G,w}(x_1, x_2, \dots, x_n)=\sum_{K \in L}\left(W(K)\prod_{i=1}^{t(K)}\left(\sum_{v \in A_i}w(v_nv)x_v\right) x_n^{t(K)-1}\right).
\]
Теперь давайте смотреть на то, как устроены деревья в графе $G_1$.
Эти деревья делятся на те, которые не содержат ребро $v_nv_{n+1}$, и те, которые содержат его. Тогда пусть множество деревьев первого типа это $S_1$, второго --- $S_2$.
Тогда
\[
P_{G_1,w_1}(x_1, \dots, x_{n+1})=\sum_{T \in S(G)} \prod_{e \in T}w(e) \prod_{v \in V} x_v^{\deg_T(v)-1}= 
\]
\[
\sum_{T \in S_1} \prod_{e \in T}w(e) \prod_{v \in V} x_v^{\deg_T(v)-1}+\sum_{T \in S_2} \prod_{e \in T}w(e) \prod_{v \in V} x_v^{\deg_T(v)-1}= P_1(x_1, \dots, x_{n+1})+P_2(x_1, \dots, x_{n+1}). 
\]
Деревья в $S_1$ устроены следующим образом. Аналогично исходному графу мы рассматриваем лес, который содержит  все вершины кроме $v_n, v_{n+1}$, и такой, что каждая его компонента содержит хотя бы одну вершину из $N_G(v_n)$, после чего соединяем одну из долей с обеими вершинами из $v_n, v_{n+1}$, а каждая из остальных $t(K)-1$ долей --- с ровно одной из этих вершин. Заметим, что с точки зрения веса дерева не важно, с какой из вершин $v_n, v_{n+1}$ мы соединяем очередную вершину, потому что по построению $w(v_nt)=w_1(v_nt)=w_1(v_{n+1}t)$ для любого $t \in N(v_n)$.
Тогда 
\[
P_1(x_1, \dots, x_{n+1})= \sum_{K \in L}\left(W(K)\sum_{i=1}^{t(K)} \left(\left(\sum_{v \in A_i} w(v_nv)x_v\right)^2\prod_{j=1, j \neq  i}^n\left(\sum_{v \in A_j}w(v_nv)x_v\right)(x_n+ x_{n+1})^{t(K)-1}\right)\right)=
\]
\[
=\sum_{K \in L}\left(W(K)\sum_{i=1}^{t(K)} \left(\left(\sum_{v \in A_v} w(v_nv)x_v\right)\prod_{j=1}^n\left(\sum_{v \in A_j}w(v_nv)x_v\right)(x_n+ x_{n+1})^{t(K)-1}\right)\right)=
\]
\[
=\left(\sum_{v \in N_G(v_n)}w(v_nv)x_v\right)\cdot\sum_{K \in L}\left(W(K)\prod_{j=1}^n\left(\sum_{v \in A_j}w(v_nv)x_v\right)(x_n+ x_{n+1})^{t(K)-1}\right).
\]
Теперь видно, что второй сомножитель это $P_{G,w}(x_1, \dots, x_n+x_{n+1})$.

Давайте разбираться с деревьями из $S_2$.
Заметим, что они устроены так: мы снова берем лес с такими же условиями, каждую из долей соединяем с ровно одной из вершин $v_n, v_{n+1}$, и соединяем их между собой.
Отметим, что снова не важно, с какой конкретно из них мы соединяем, причем по тем же причинам.
Тогда 
\[
P_2(x_1, \dots x_{n+1})=p\sum_{K \in L}\left(W(K)\prod_{i=1}^{t(K)} \left (\sum_{v \in A_i}w(v_nv)x_v \right) ( x_n+x_{n+1})^{t(K)}\right)=
\]
\[
= p(x_n+x_{n+1}) P_{G,w}(x_1, x_2, \dots, x_n+x_{n+1}).
\]
Отметим, что $p$ в начале берется из рассматриваемого ребра $v_nv_{n+1}$, а множители $x_n$ и $x_{n+1}$ сокращаются с $-1$ в степенях вхождения переменных.
Тогда получается, что
\[
P_{G_1,w_1}(x_1, \dots, x_{n+1})=P_1(x_1, \dots, x_{n+1})+ P_2(x_1, \dots, x_{n+1})=
\]
\[
=P_{G,w}(x_1, \dots, x_n)\left (\sum_{v \in N(v_n)}w(v_nv)x_v \right)+ P_{G,w}(x_1, \dots, x_n) \cdot (px_n +px_{n+1})=\]
\[
P_{G,w}\left(\sum_{v \in N(v_n)}w(v_nv)x_v) +p(x_n+ px_{n+1}) \right).
\]
\end{proof}

\begin{lemma}[о точках сочленения]  \label{lm:gluing}
Пусть дан взвешенный граф $G$, имеющий точку сочленения $v$, и пусть множество вершин графа $G$ после удаления $v$ распадается на компоненты связности с множествами вершин $V_1,V_2, \dots, V_k$. Положим $G_i = G[V_i \cup \{v\}]$. 
Тогда 
\[
P_G(x_1, x_2, \dots, x_n)=\prod_{i=1}^kP_{G_i}x_v^{k-1}, 
\]
где в $P_{G_i}$ подставляются переменные, соответствующие вершинам, принадлежащим $G_i$.
\end{lemma}
\begin{proof}
Очевидно, что каждое остовное дерево в графе $G$ однозначно соответствует набору остовных деревьев в $\{G_i\}$, а степень вершины $v$ это в точности сумма её степеней в остовных деревьях $G_i$, что соответствует умножению многочленов.
\end{proof}

\begin{lemma}[о домножении]  \label{lm:multyplying}
Пусть $P_{G,w}$ --- стабильный многочлен взвешенного графа $G$.
Тогда домножение всех ребер, инцидентных некоторой вершине $v$, на вещественное $c>0$ не меняет взвешенную стабильность $G$.
\end{lemma}

\begin{proof}
Новый многочлен будет стабильным тогда же, когда и старый, поскольку домножение соответствует замене переменной $x_v$ в $v$, на $cx_v$:
\[
P_{new(v,c)}(x_1,x_2,\dots, x_v, \dots, x_n)=\frac{P(x_1, \dots, cx_v, \dots, x_n)}{c}, 
\]
и $cx_v\in \mathbb{H}$ тогда и только тогда, когда $x_v\in \mathbb{H}$.
\end{proof}

\subsection{Любой взвешенно-стабильный граф получается операциями}

По лемме~\ref{lm:multyplying} мы можем домножать все ребра вокруг одной вершины на положительную константу, приводя весовую функцию к удобному виду.

\begin{lemma}
\label{lm:support}
Пусть для некоторого графа $G$ и функции весов на нем $w$ многочлен $P_{G,w}$ стабильный. Тогда $G$ стабильный как невзвешенный граф.
\end{lemma}

\begin{proof}
Предположим противное. Тогда $G$ не является стабильным как невзвешенный граф, то есть по теореме~\ref{th:main} не является дис\-тан\-ци\-он\-но-наследуемым, значит по эквивалентному определению~(v) $G$ содержит длинный цикл, домино, дом или самоцвет как индуцированный подграф. Разберем эти случаи, и покажем, что в каждом из них путем нескольких операций домножения всех ребер, инцидентных вершине на положительное число можно сделать веса на индуцированном подграфе единичными. Лемма~\ref{lm:multyplying} и утверждение~\ref{pr:basic}~(iii) дают противоречие.

\paragraph{Циклы.}
Цикл длины пять сводится к циклу, в котором все ребра имеют веса $1$ при помощи операций домножения.
Рассмотрим цикл длины хотя бы шесть, не имеющих хорд.
Домножим во всех вершинах кроме одной последовательно, так, чтобы веса всех ребер, кроме одного, стали $1$. Если последнее ребро это $v_1v_n$, то вес всех деревьев, кроме того, которое получается при удалении $v_1v_n$ это $w(v_1v_n)$.
Тогда можно просто переписать доказательство для обычного цикла без весов, потому что слагаемое, которое отличается, все равно в нашей подстановке оказывается нулем.

\paragraph{Домино.}
Пусть ребро, не входящее в цикл $v_1v_2 \dots v_6v_1$, это $v_1v_4$ (см. рис.~\ref{fig:forbidden}).
Последовательными применяя лемму~\ref{lm:multyplying}, добьемся равенств
\[
1=w(v_2v_1)=w(v_1v_6)=w(v_6v_5)=w(v_5v_4)=w(v_4v_3).
\]
Тогда рассмотрев набор $v_1,v_6,v_5,v_4$ получим, что $w(v_4v_1)=1$; аналогично для набора $v_2v_1v_5v_3$ имеем, что $w(v_2v_3)=1$. Получилось обычное (невзвешенное) домино, а оно нестабильно. 

\paragraph{Дом.}
Применим лемму~\ref{lm:multyplying} так, чтобы ребра цикла получили вес 1. Теперь рассмотрим подграф $C_4$, образованный вершинами $v_1v_2v_4v_5$.

Покажем, что в стабильном $C_4$ можно нормировками сделать все веса равными.
Давайте за три нормировки сделаем веса всех ребер, кроме одного, единицами.
Необходимо показать, что последнее ребро (пусть это $v_4v_1$) имеет вес $1$.
При этом наш многочлен это 
\[
x_2x_3+w(v_1v_4)(x_1x_2+x_3x_4+x_1x_4).
\]
Пусть  $w(v_1v_4)=\frac{1}{t}$, тогда наш многочлен это
\[
\frac{1}{t}(tx_2x_3+x_1x_2+x_3x_4+x_1x_4).
\]
Если $t>1$, то рассмотрим подстановку $x_3=x_4=1$. Тогда получится, что
\[
x_2(tx_1+1)+x_1+1=0.
\]
Но у этого многочлена есть корень  
\[
x_1=i,  \quad \quad x_2=\frac{-(1+i)(1-ti)}{t^2+1}=\frac{-(t+1)+(t-1)i}{t^2+1},
\]
что противоречит стабильности. Второй случай ($t<1$) разбирается аналогично.
Получилось, что $t=1$, что мы и хотели.

Значит перед нами невзвешенный дом, а он не стабилен.

\paragraph{Самоцвет.} Обозначим $v$ вершину степени 4.
Давайте отнормируем так, чтобы все веса всех ребер, инцидентных $v$ стали $1$. Тогда оставшиеся три ребра одинаковые, а значит если рассмотреть индуцированный граф на множестве вершин без $v$, получится, что эти одинаковые ребра имеют вес $1$, то есть перед нами обычный самоцвет (потому что получается, что в множестве $\{x,x^2,0\}$ есть два равных числа, и $x\neq 0$). Но обычный самоцвет нестабилен. 

\end{proof}

Вернемся к доказательству теоремы~\ref{mainweighted}. 
Заметим, что если граф $G$ не двусвязен, то достаточно доказать утверждение теоремы для каждой компоненты двусвязности.

Назовем пару вершин $x_1$ и $x_2$ \textit{стягиваемой}, если их окрестности в графе $G$ совпадают (возможно по модулю них самих) и отношение
\[
\frac{w(vx_1)}{w(vx_2)}
\]
одинаковое для всех вершин $v$, соединенных с $x_1$ и $x_2$. 
Заметим, что домножение весов всех ребер, инцидентных вершине, на положительную константу не меняет множество стягиваемых пар графа. 
В доказательстве мы будем часто пользоваться этим свойством, применяя лемму~\ref{lm:multyplying}.

Мы докажем по индукции следующее, несколько более сильное, утверждение.

\begin{proposition}
Пусть для некоторого двусвязного графа $G$ на хотя бы четырех вершинах и функции весов на нем $w$ многочлен $P_{G,w}$ стабильный. 
Тогда в $G$ найдутся хотя бы две непересекающиеся стягиваемые пары.
\label{mainProp}
\end{proposition}

\subsection{База для четырех вершин}
\label{subsec:baza}

Двусвязные графы на четырех вершинах это полный граф $K_4$, цикл $C_4$ и цикл с диагональю $C_4^+$.

\paragraph{Случай $C_4$} разобран при рассмотрении дома в предыдущем разделе.

\paragraph{Случай $K_4$.} Путем домножений в вершинах 1, 2 и 3 можно добиться того, чтобы ребра из вершины 4 имели единичный вес.
По лемме~\ref{lm:multyplying} данные операции не влияют на свойство стабильности графа.
Обозначим $e_1 = w(2,3)$, $e_2 = w(1,3)$, $e_3 = w(1,2)$. В новых обозначениях 
\[
P_{K_4,w}=x_4^2+x_4(x_1(e_2+e_3)+x_2(e_1+e_3)+x_3(e_1+e_2))+ (x_1+x_2+x_3)(x_1e_2e_3+x_2e_1e_3+x_3e_1e_2).
\]
Выражение является квадратичным относительно переменной $x_4$, а значит при любой вещественной подстановке $x_1,x_2,x_3$ дискриминант $D$ соответствующего квадратного трехчлена хотя бы $0$. Имеем 
\[
0 \leq (x_1(e_2+e_3)+x_2(e_1+e_3)+x_3(e_1+e_2))^2-4(x_1+x_2+x_3)(x_1e_2e_3+x_2e_1e_3+x_3e_1e_2)=
\]
\[
=x_1^2(e_2-e_3)^2+x_2^2(e_3-e_1)^2+x_3^2(e_1-e_2)^2-2x_1x_2(e_2-e_3)(e_3-e_1)-2x_2x_3(e_3-e_1)(e_1-e_2)-2x_3x_1(e_1-e_2)(e_2-e_3)
\]
при любых вещественных $x_1,x_2,x_3$. Если никакие две переменные из $e_1,e_2,e_3$ не совпали, то мы можем подставить 
\[
x_1=\frac{1}{e_2-e_3}, x_2=\frac{1}{e_3-e_1}, x_3=\frac{1}{e_1-e_2},
\]
что дает $D =-3$. Это противоречит неотрицательности дискриминанта. Значит из трех величин $e_1,e_2,e_3$ какие-то две совпали, что влечет совпадение каких-то двух произведений противоположных ребер. 
Утверждение для $K_4$ доказано.

\paragraph{Случай $C_4^+$} аналогичен предыдущему, если подставить $e_1 = 0$.

\subsection{Переход}

\begin{lemma}
Пусть $G$ и $w$ --- граф и функция весов на нем. Пусть также в $G$ найдутся два индуцированных подграфа $G_1, G_2$, в объединении покрывающие $G$, и некоторые три вершины $x_1,x_2,y$, принадлежащие обоим подграфам, а так же в $G$ (а значит и в $G_1$ и в $G_2$) есть ребра $x_1y$ и $x_2y$, и пара $x_1x_2$ стягиваема и в $G_1$, и в $G_2$. Тогда пара $x_1x_2$ стягиваема и в $G$.
\label{lm:covering}
\end{lemma}

\begin{proof}
Поскольку подграфы $G_1$ и $G_2$ содержат все вершины $G$, вершины $x_1$ и $x_2$ графски эквивалентны в $G$. Рассмотрим произвольную вершину $v$, соединенную с $x_1$ и $x_2$; для нее выполняется
\[
\frac{w(vx_1)}{w(vx_2)}=\frac{w(yx_1)}{w(yx_2)}, 
\]
так как $v$ принадлежат $G_1$ или $G_2$, а в них пара $x_1x_2$ стягиваема. Таким образом, отношение $\frac{w(vx_1)}{w(vx_2)}$ не зависит от выбора $v$. 
\end{proof}

Пусть $G$ --- граф, и $w$ --- весовая функция на нем, такие что $P_{G,w}$ --- стабильный. По лемме~\ref{lm:support} и теореме~\ref{th:main} граф $G$ --- дистанционно-наследуемый. По определению~(vi) класса дистанционно-наследуемых графов $G$ можно построить путем копирования и добавления висячих вершин, и в силу двусвязности $G$ последней операцией было копирование. Это значит, что в графе $G$ есть вершины $v$ и $v'$, такие, что их окрестности совпадают (по модулю самих вершин $v$ и $v'$), то есть
\[
N_G(v) \setminus \{v'\} = N_G(v') \setminus \{v\}.
\]

Рассмотрим взвешенный граф $H = G[V\setminus \{v'\}]$.
По утверждению~\ref{pr:basic}(iii) и из-за двусвязности $G$, граф $H$ --- взвешенно-стабильный, так как удаление $v'$ соответствует подстановке $x_{v'} = 0$.

\paragraph{Случай недвусвязного $H$.}  Двусвязность $G$ влечет, что в $H$ не более одной вершины сочленения, и эта вершина это $v$.
Тогда давайте докажем, что $v$ и $v'$ стягиваемы в $G$ с учетом веса.
Действительно, пусть окрестности $v$ в разных компонентах двусвязности $H$ это $A_1, A_2, \dots, A_k$. Напомним, что операция домножения не влияет на условие стягиваемости вершин и отнормируем все так, чтобы все ребра инцидентные $v$, кроме, возможно, ребра из $v'$, имели вес $1$. Покажем, что все веса для вершины $v'$, кроме, возможно, $w(vv')$, равны между собой.
Для этого достаточно показать, что для произвольных смежных с $v$ и $v'$ вершин $x \in A_i, y \in A_j, i \neq  j$ верно равенство $w(v'x) = w(v'y)$. Для этого рассмотрим четверку $v, v', x, y$. Тогда в силу того, что $x$ и $y$ в разных компонентах двусвязности, $G$ не содержит ребра $xy$, а значит эта четверка это либо $C_4$, либо $C_4^+$. Оба случая рассмотрены в базе и в них стягиваемые пары это $vv'$ и $xy$.
Из стягиваемости пары $xy$ следует нужное нам равенство.
Таким образом для недвусвязных $H$ одну стягиваемую пару в $G$ мы нашли. 

Пусть компоненты двусвязности в $H$ это $H_1, H_2, \dots, H_m$, $m \geq 2$ (они все содержат $v$).
Тогда рассмотрим следующие случаи: 
\begin{itemize}
    \item Существует такое $k$, что $|H_k| \geq 4$. Тогда поскольку $|H_k| < |H|$ из предположения индукции следует, что в $H_k$ есть две непересекающиеся стягиваемые пары, тогда одна из них не содержит $v$, значит она стягиваема и в $G$. Действительно, они остались одинаковыми с точки зрения графа, и даже если при копировании появилась новая вершина в окрестности, то это только $v'$, и в силу того, что $vv'$ --- стягиваемая пара в $G$, по лемме~\ref{lm:copy} вершина $v'$ не портит взвешенную стабильность. 
    \item Найдется такое $k$, что $|H_k|=3$. Тогда $H_k$ --- цикл $vu_1u_2$ на трех вершинах, то есть вершины $u_1$ и $u_2$ имеют совпадающую окрестность (за исключением друг друга), состоящую из единственной вершины $v$. То есть $u_1u_2$ --- стягиваемая пара в $H_k$, а значит в $H$ и $G$ по тем же причинам.
    \item Каждый из графов $H_k$ содержит ровно две вершины, то есть каждое $H_i$ является ребром; обозначим его $u_iv$. Тогда в $G$ есть стягиваемая пара $u_1u_2$, потому что у них совпадают окрестности --- это вершины $v,v'$, и $\frac{w(u_1v)}{w(u_2v)}=\frac{w(u_1v')}{w(u_2v')}$ в силу стягиваемости $vv'$ в $G$.
\end{itemize}   
Таким образом, если граф $H$ получился не двусвязным, то теорема доказана.

\paragraph{Случай двусвязного $H$} значительно труднее.

\begin{lemma}
Утверждение~\ref{mainProp} справедливо для полного взвешенного графа $G$ и функции весов $w$ на нем.
\label{lm:complete}
\end{lemma}

\begin{proof}
Несложно видеть, что все графы на не более чем трех вершинах являются стабильными. 
Оставшееся доказательство является индуктивным, где база для четырех вершин доказана в разделе~\ref{subsec:baza}.

Переход.  Предположим, противное, тогда $G$ содержит не более одной стягиваемой пары.

Разберем случай, когда $G$ вообще не содержит стягиваемых пар.
Давайте удалим произвольную вершину $x$. 
По утверждению~\ref{pr:basic}(iii) оставшийся взвешенный полный граф $F$ является вещественно-стабильным, значит по предположению индукции в $F$ есть две непересекающиеся стягиваемые пары, пусть это $v_1v_2$ и $u_1u_2$.
Применим несколько раз лемму~\ref{lm:multyplying}, так чтобы $w(u_1u_2) = w(v_1u_2) = w(v_1u_1) = w(v_2u_2) = 1$ (для этого нужно сначала приравнять за три нормировки веса ребер, ведущих в $u_2$, а потом приравнять к ним вес в оставшемся ребре за счет нормировки в $u_2$). Тогда $w(v_2u_1) = 1$ за счет стягиваемости пары $u_1u_2$ в $F$. 

Получаем, что для любой вершины $q$, которую мы еще не рассмотрели, стягиваемость пар $u_1u_2$ и $v_1v_2$ в графе $F$ влечет $w(qu_1)=w(qu_2)$ и $w(qv_1)=w(qv_2)$.

Мы предположили, что в исходном графе ни одна из пар $u_1u_2$ и $v_1v_2$ не стягивается, а значит $w(xv_1) \neq w(xv_2)$ и $w(xu_1) \neq w(xu_2)$. Рассмотрим графы $G[\{x,v_1,u_1, u_2\}]$ и $G[\{x, v_2, u_1, u_2\}]$. Они стабильные, а значит из трех чисел $w(xv_1), w(xu_1), w(xu_2)$ есть два равных, и из $w(xv_2), w(xu_1),w(xu_2)$ тоже. Тогда получается, что либо $w(xv_1)=w(xu_1)$ и $w(xv_2)=w(xu_2)$, либо $w(xv_1)=w(xu_2)$ и $w(xv_2)=w(xu_1)$. Не умаляя общности, можно считать, что мы имеем дело с первым случаем.

Давайте покажем, что $w(v_1v_2)=1$.
Предположим, что $w(v_1v_2)= t \neq  1$, и пускай $w(xv_1)=w(xu_1)=a$ и $w(xv_2)=w(xu_2)=b$. Рассматривая графы $G[\{x,v_1,v_2, u_1\}]$ и $G[\{x, v_1, v_2, u_2\}]$ имеем, что среди $b, a, at$ есть равная пара, и среди $a, b, bt$ есть равная пара. Но заметим, что поскольку $a \neq  b$ и $t \neq 1$, получается, что $a = bt$ и $b = at$. Но так не бывает. Противоречие, значит $w(v_1v_2)=1$.

Тогда заметим, что поскольку в нашем графе нельзя стянуть $u_1v_1$, есть вершина $y$, для которой $w(yu_1)\neq w(yv_1)$. Тогда $y$ это новая вершина, потому что для $x, v_2, u_2$ эти веса равны. Тогда мы знаем, что $w(yv_1)=w(yv_2)$ и $w(yu_1)=w(yu_2)$.

Применяя лемму~\ref{lm:multyplying} в вершинах $x$ и $y$ так, чтобы $w(xv_1)=w(xu_1)=1$ и $w(yu_1)=w(yu_2)=1$. Тогда все веса зависят от трех переменных: $a=w(xv_2)=w(xu_2)$, $b=w(yv_1)=w(yv_2)$, $c=w(xy)$, и также мы знаем, что $a,b \neq 1$.
Рассматривая индуцированные подграфы на всевозможных четверках вершин, состоящих из $x, y$ и еще двух вершин из $\{u_1,v_1,u_2,v_2\}$, получаем, что в каждом из следующих мультимножеств есть равные числа:
\[
\{a,c,1\}, \{b,c,1\}, \{a,b,c\}, \{1,ab,c\}, \{c,a,ab\}, \{c,b,ab\}.
\]
Из первых двух множеств имеем, что либо $c=1$, либо $a=b=c$.
В первом случае из третьего мультимножества имеем, что $a=b$, и тогда пятое мультимножество $\{a, a^2, 1\}$ содержит два равных числа.
Во втором случае четвертого мультимножество $\{1,a,a^2\}$ содержит два равных числа.
В обоих случаях мы получаем противоречие, так как $a \notin\{0,1\}$.

Осталось рассмотреть случай, когда $G$ содержит ровно одну стягиваемую пара $v_1v_2$. 
Применяя лемму~\ref{lm:multyplying}, отнормируем ребра инцидентные $v_2$, так, чтобы $w(uv_2)=w(uv_1)$ для какой-нибудь вершины $u$. Тогда $w(uv_2)=w(uv_1)$ для любой вершины $u$, потому что пара  $v_1v_2$ стягиваемая.
По предположению индукции в графе $G[V\setminus \{v_2\}]$ отыщутся две непересекающиеся стягиваемые пары, а значит одна из них не содержит $v_1$. Пусть эта пара это $u_1u_2$. Тогда заметим, что она стягиваема и в исходном графе, потому что значение $\frac{w(u_1x)}{w(u_2x)}$ одинаковое при всех $x \neq v_1$, а также совпадает при $x = v_1$ и $x = v_2$. Тогда мы нашли две непересекающиеся стягиваемые пары в графе $G$, как и хотели.

\end{proof}

Перейдем к произвольному $G$. Сначала найдем в $G$ одну стягиваемую пару.

Поскольку $H$ двусвязный, к нему применимо предположение индукции, то есть в $H$ найдутся две непересекающиеся стягиваемые пары $x_1x_2$ и $y_1y_2$. Тогда вершина $v$ либо принадлежит одной из этих пар, либо не принадлежит.

\paragraph{Случай 1.}
Вершина $v$ не принадлежит ни одной из двух выделенных стягиваемых пар $H$.
Поскольку $N_G(v) \setminus \{v'\} = N_G(v') \setminus \{v\}$ и $N_H(x_1) \setminus \{x_2\} = N_H(x_2) \setminus \{x_1\}$, ребра $x_1v$, $x_1v'$, $x_2v$, $x_2v'$ либо все есть в $G$, либо ни одного из них нет в $G$. Если их всех нет, то стягиваемость $x_1x_2$ в $H$ влечет стягиваемость в $G$ и искомая пара найдена. 
В оставшемся случае $G$ содержит эти четыре ребра, и аналогично все ребра $y_1v$, $y_1v'$, $y_2v$, $y_2v'$ также есть в $G$.
Если $G$ не содержит ребра $vv'$, то база для графа $G[\{x_1,x_2,v,v'\}]$ дает стягиваемые пары $x_1x_2$ и $vv'$, потому что на этой четверке вершин есть цикл, и нет ребра $vv'$; и тогда по лемме~\ref{lm:covering} пара $x_1x_2$ стягиваема в $G$.
Аналогично, рассматривая графы $G[\{x_1,x_2,v,v'\}]$ и $G[\{y_1,y_2,v,v'\}]$, получаем, что ребра $x_1x_2$ и $y_1y_2$ есть в $G$.

Теперь покажем, что $G$ содержит все ребра вида $x_iy_j$, $i,j \in \{1,2\}$. Аналогично предыдущему абзацу $G$ содержит либо их все, либо ни одного из них.
Предположим, что данных ребер нет, и рассмотрим граф $G[\{v,v',x_i,y_j\}]$. В нем есть цикл на всех вершинах, и нет ребра $x_iy_j$, а значит в нем стягиваемы пары $vv'$ и $x_iy_j$.
Применяя определение стягиваемости, получаем
\[
\frac{w(vx_i)}{w(v'x_i)}=\frac{w(vy_j)}{w(v'y_j)}
\]
для $i=1$, $j=1,2$ получаем, что $y_1y_2$ является стягиваемой парой в $G[\{vv'y_1y_2\}]$, из чего по лемме~\ref{lm:covering} пара $y_1y_2$ стягиваема в $G$. 
Значит все ребра вида $x_iy_j$ есть в $G$.

Тогда граф $F = G[\{x_1,x_2,y_1,y_2,v'\}]$ полный, а значит по лемме~\ref{lm:complete} в нем есть две непересекающиеся стягиваемые пары. Тогда если одна из этих пар содержит $v'$, то, не умаляя общности, это пара $v'x_1$, тогда $\frac{w(x_1y_1)}{w(v'y_1)}=\frac{w(x_1y_2)}{w(v'y_2)}$, откуда $\frac{w(x_1y_1)}{w(x_1y_2)}=\frac{w(v'y_1)}{w(v'y_2)}$, а значит пара $y_1y_2$ стягиваема в $G$.
Значит в $F$ стягиваемы две непересекающиеся пары из множества $\{x_1,x_2,y_1,y_2\}$. Если это пары $x_1x_2$ и $y_1y_2$, то $\frac{w(x_1y_1)}{w(x_2y_1)}=\frac{w(x_1v')}{w(x_2v')}$, и пара $x_1x_2$ стягиваема в $G$, что завершило бы разбор случая 1.
Не умаляя общности, осталась ситуация, в которой стягиваемыми парами в $F$ являются $x_1y_1$ и $x_2y_2$. Тогда в графе 
$G[\{x_1,x_2,y_1,y_2\}]$ есть четыре стягиваемые пары вершин, а значит, применяя лемму~\ref{lm:multyplying}, отнормируем веса так, чтобы все шесть ребер среди этих четырех вершин имели вес $1$. Из стягиваемости пар $x_1x_2$ и $y_1y_2$ в $H$ для любой вершины $u \notin \{v',x_1,x_2\}$ ребра $ux_1$ и $ux_2$ либо оба отсутствуют в $G$, либо оба есть в $G$, и тогда $w(ux_1) = w(ux_2)$ и для любой вершины $u \notin \{v',x_1,x_2\}$ ребра $uy_1$ и $uy_2$ либо оба отсутствуют в $G$, либо оба есть в $G$, и тогда $w(uy_1) = w(uy_2)$. Также из стягиваемости пар $x_1y_1$ и $x_2y_2$ в $F$ следуют равенства $w(v'y_1) = w(v'x_1)$ и $w(v'y_2) = w(v'x_2)$.

Поскольку $x_1x_2$ и $y_1y_2$ не являются стягиваемыми в $G$, $w(v'x_1) \neq w(v'x_2)$.
Рассмотрим пару $x_1y_1$. Она не стягиваема в $G$, следовательно либо найдется вершина $u \in (N_G(x_1) \Delta N_G(y_1)) \setminus \{x_1,y_1\}$, либо найдется вершина $z \in N_G(x_1) \cap N_G(y_1)$, такая что $w(zx_1) \neq w(zy_1)$ (учитывая нормировку из прошлого абзаца, делающую веса в подграфе $G[\{x_1,x_2,y_1,y_2\}]$ единичными).

В первом варианте, не умаляя общности, $ux_1 \in E, uy_1 \notin E$. 
Тогда в силу $N_H(x_1) \setminus \{x_2\} = N_H(x_2) \setminus \{x_1\}$ имеем $ux_2 \in E, uy_2 \notin E$. Тогда заметим, что если в графе $G$ нет ребра $uv'$, то в графе $G[\{v',x_1,x_2, u\}]$ есть цикл, и нет ребра $uv'$, а значит стягиваемая пара это $x_1x_2$, но это противоречит тому, что $w(x_1u)=w(x_2u)$ (т.к. $w(y_1x_1)=w(y_1x_2)$ и $x_1x_2$ стягиваемо в $H$) и $w(v'x_1) \neq w(v'x_2)$. 
Значит ребро $uv'$ принадлежит $G$. Тогда нужно посмотреть на графы $G[\{u, v', x_1, y_1\}]$ и $G[\{u, v', x_1, y_2\}]$. В каждом из них стягиваемая пара это $v'x_1$, потому что есть цикл и нет ребра $uy_j$. Мы получили противоречие, поскольку 
\[
w(y_1v')=\frac{w(y_1v')}{w(y_1x_1)}=\frac{w(uv')}{w(ux_1)}=\frac{w(y_2v')}{w(y_2x_1)}=w(y_2v').
\]

Во втором варианте ребра $zx_2, zy_2$ тоже есть в $G$, и их веса равны соответственно $w(zx_1)$ и $w(zy_1)$, потому что $z \neq v'$, а значит $z \in H$. Также $z \notin \{x_1,x_2, y_1,y_2,v'\}$.
Домножим все ребра, инцидентные вершинах $v'$ и $z$ так, чтобы $w(v'x_1)=w(zx_1)=1$ (это не меняет веса внутри $G[\{x_1,x_2,y_1,y_2\}]$), что влечет $w(v'y_1)=w(zx_2)=1$. Положим $a=w(zy_1)=w(zy_2)$, и $b=w(v'x_2)=w(v'y_2)$. Определение $z$ влечет $a \neq 1$, а нестягиваемость $x_1$ и $x_2$ в $G$ дает $b \neq 1$. 
Заметим, что если в $G$ нет ребра $zv'$, то в графе $G[\{x_1,x_2,v',z\}]$ есть пять ребер из шести, а значит пара $v'z$ должна быть стягиваема в этом графе, но $w(x_1z)=1=w(x_2z)$, а $w(x_2v') = b \neq  1 = w(x_1v')$. Противоречие, значит в $G$ есть ребро $zv'$. Тогда к графу $J = G[\{x_1,x_2, y_1,y_2,z, v'\}]$ применима лемма~\ref{lm:complete} и в нем есть две стягиваемые пары. 
Ни одна из пар внутри множества $U := \{ x_1,x_2,y_1,y_2\}$ не стягиваема, потому что у любой пары, кроме $x_1x_2$ и $y_1y_2$ разные веса на ребрах к вершине $z$, а у пар $x_1x_2$ и $y_1y_2$ разные веса на ребрах к вершине $v'$. 
Тогда каждая из стягиваемых пар в $J$ содержит по вершине не из $U$, значит одна из них $v't$, где $t \in U$.
Тогда $t$ не принадлежит какой-то из пар $x_1x_2$, $y_1y_2$ (не умаляя общности, $x_1x_2$), и тогда рассмотрение индуцированного подграфа $G[\{t,v',x_1,x_2\}]$ дает стягиваемые в нем пары $v't$ и $x_1x_2$, а значит по лемме~\ref{lm:covering} пара $x_1x_2$ была стягиваема и в $G$.  
Получается, что со случаем, когда $v$ не принадлежит ни одной стягиваемых пар в $H$ мы разобрались.

\paragraph{Случай 2.}
Теперь пусть $v$ принадлежит одной из пар $x_1x_2$ и $y_1y_2$; не умаляя общности, $v = x_1$. Напомним, что по определению
$N_G(v) \setminus \{v'\} = N_G(v') \setminus \{v\}$, а также $N_H(x_1) \setminus \{x_2\} = N_H(x_2) \setminus \{x_1\}$ и 
$N_H(y_1) \setminus \{y_2\} = N_H(y_2) \setminus \{y_1\}$.
Тогда следующие шесть ребер либо все есть в $G$, либо их всех нет: $x_1y_1,x_1y_2,x_2y_1,x_2y_2,v'y_1,v'y_2$. Если их всех нет, то пара $y_1y_2$ стягиваема в $G$.

Остался случай, когда все эти ребра есть.
Рассматривая графы $G[\{x_1,v',y_1,y_2\}]$ и $G[\{x_2,v',y_1,y_2\}]$, мы видим, что либо в них стягиваема пара $y_1y_2$ (а тогда она по лемме~\ref{lm:covering} стягиваема в $G$), либо в $G$ есть ребра $v'x_1$, $v'x_2$ и $y_1y_2$. Аналогично, глядя на граф $G[\{v',x_1,x_2,y_1\}]$, мы видим, что либо ребро $x_1x_2$ есть в $G$, либо $x_1x_2$ --- стягиваемая пара. Таким образом если мы не нашли стягиваемую пару, то подграф $G[\{x_1,x_2,y_1,y_2,v'\}]$ полный, а значит по лемме~\ref{lm:complete} в нем стягиваемы две непересекающиеся пары вершин. Тогда одна из них не содержит $v'$, и при этом она не равна $x_1x_2$ и $y_1y_2$, потому что в таком случае эта пара была бы стягиваема в $G$ по лемме~\ref{lm:covering}. Более того, ни одна из этих пар не содержит $v'$, потому что иначе можно рассмотреть индуцированный граф на этой паре и непересекающейся с ней паре из $x_1x_2$, $y_1y_2$; тогда соответствующая пара из $x_1x_2$, $y_1y_2$ окажется стягиваемой в $G$ по лемме~\ref{lm:covering}. 

Не умаляя общности, осталась ситуация, когда стягиваемыми парами в $G[\{x_1,x_2,y_1,y_2,v'\}]$ являются $x_1y_1$ и $x_2y_2$.
Дальнейшие рассуждения совпадают с окончанием разбора случая 1: лемма~\ref{lm:multyplying} позволяет за счет четырех стягиваемых пар в $G[\{x_1,x_2,y_1,y_2\}]$ добиться единичных весов внутри $G[\{x_1,x_2,y_1,y_2\}]$, а последующее домножение всех ребер, инцидентных  вершине $v'$, дает $1 = w(v'x_1) = w(v'y_1)$ и $a = w(v'x_2) = w(v'y_2) \neq 1$.
Пара $y_1$ и $x_1$ не стягиваема в $G$, значит существует вершина $z$, которая их различает эти вершины.
Тогда либо $z \in (N_G(x_1) \Delta N_G(y_1)) \setminus \{x_1,y_1\}$, либо  $z \in N_G(x_1) \cap N_G(y_1)$ и $w(zx_1) \neq w(zy_1)$.

Предположим, что она делает это графски. Есть два подслучая, в первом $z$ соединена с $x_1,x_2,v'$; рассмотрим графы 
$G[\{z,v',x_1,y_1\}]$ и $G[\{z,v',x_1,y_2\}]$. В них стягиваемыми парами являются $zx_1$ и $v'y_i$, а тогда 
 \[
 w(y_1v')=\frac{w(y_1v')}{w(y_1x_1)}=\frac{w(zv')}{w(zx_1)}=\frac{w(y_2v')}{w(y_2x_1)}=w(y_2v'),
 \]
и $y_1y_2$ стягиваема в $G$. Во втором подслучае, 
%когда $z$ различает графски такой: 
$z$ соединена с $y_1,y_2$, и тогда рассмотрим граф $G[\{z,v',y_2,y_1\}]$. В нем нет ребра $v'z$, а значит в этом графе пары $v'z$ и $y_1y_2$ стягиваемы. По лемме~\ref{lm:covering} $y_1y_2$ стягиваема во всем $G$.

Во втором варианте $z \notin \{x_1,x_2,y_1,y_2,v'\}$ и в $G$ существуют все ребра, ведущие из $z$ в уже рассмотренную пятерку вершин $x_1x_2y_1y_2v'$, то есть к графу $G[\{x_1,x_2,y_1,y_2,v',z\}]$ применима лемма~\ref{lm:complete}. 

\paragraph{Завершение доказательства.} В обоих случаях мы нашли стягиваемую пару $xy$ в $G$. Давайте вернемся в начало доказательство и рассмотрим эту пару в качестве пары $vv'$, а именно посмотрим на граф $H' = G[V\setminus \{x\}]$. 
Случай не двусвязного $H'$ разобран в начале доказательства. 
Если же $H'$ двусвязен, то по предположению индукции в нем найдутся две непересекающиеся стягиваемые пары; рассмотрим такую из них, которая не содержит $x$ и назовем ее $uv$. 
Тогда $xy$ и $uv$ --- искомые непересекающиеся стягиваемые пары в $G$. Доказательства утверждения~\ref{mainProp} и теоремы~\ref{mainweighted} закончены.

\section{Взвешенные дистанционно-наследуемые графы}
\label{diss}

По аналогии с теоремой~\ref{th:main} определим класс взвешенных дис\-тан\-ци\-он\-но-наследуемых графов как класс взве\-шен\-но-стабильных графов.
Покажем, что данное определение естественно. По лемме~\ref{lm:support} носитель взвешенного дис\-та\-ци\-он\-но-наследуемого графа является дистанционно-наследуемым графом. Далее, покажем, что наше определение совпадает еще с одним алгебраическим определением, данным в статье~\cite{oum2005rank}.

Обозначим ранг матрицы $M$ через $\rk(M)$.
Пусть дан взвешенный граф $G$, и его матрица смежности --- $M_G$;
 для произвольных подмножеств $A,B \in V(G)$ определим $\rk_G(A,B) := \rk(M_G[A,B])$, где $M[A,B]$ --- подматрица, с множеством столбцов, соответствующим $A$ и множеством строк, соответствующим $B$.

 Положим $\cutrk(A)=\rk_G(A, V(G) \setminus A )$; ясно что $\cutrk(A) = \cutrk(V(G) \setminus A)$, так как матрица $M_G$ является симметричной.

Рассмотрим произвольное дерево $T$, у которого висячими вершинами являются в точности вершины $G$, а все остальные вершины имеют степень $3$.
 Любое ребро этого дерева $e$ задает разбиение вершин $G$ на два множества $A(e)$ и $B(e)$, соответствующие компонентам связности в $T \setminus e$. Определим \textit{ширину ребра} $e$ как $\cutrk(A(e)) = \cutrk(B(e))$. 
 Назовем \textit{шириной графа} минимум по всем деревьям $T$ максимальной ширины ребра в нем.

В статье~\cite{oum2005rank} показано, что класс графов (не взвешенных, т.е. в которых все веса равны $1$) с шириной $1$ совпадает с дистанционно-наследуемыми.

\begin{theorem} \label{th:onemoredef}
Класс взвешенных дистанционно-наследуемых графов совпадают с классом взвешенных графов с шириной $1$.    
\end{theorem}
\begin{proof}
    
Везде далее считаем, что во всех графах хотя бы $3$ вершины, и граф связен. Отметим, что ширина связного графа не меньше, чем $1$ (иначе существует разбиение его вершин на $2$ множества, между которыми нет ребер).

\paragraph{В одну сторону.}
Пусть дан взвешенно-стабильный граф. Покажем, что он ширины $1$. Для этого нужно предъявить дерево, в котором ширины всех ребер равны $1$. 
Будем строить его по индукции по $n=|V|$.

\subparagraph{Базой} являются случаи $n \leq 3$.
Если в графе не больше $3$ вершин, то дерево всего одно и оно подходит, потому что в любом разбиении хотя бы одна из долей имеет размер не более $1$, а значит ранг не больше $1$.

\subparagraph{Переход.}  В графе $G$ хотя бы четыре вершины, значит по теореме~\ref{mainweighted} взвешенно-стабильный граф $G$ либо недвусвязен, либо содержит стягиваемую пару.

\textit{Недвусвязный случай.}
  В этом случае $G$ содержит точку сочленения $u$, то есть $G [V\setminus u]$ распадается на компоненты связности $A_1, \dots, A_k$. 
  Пусть $H_1 = G[A_1 \cup u]$, $H_2 = G[V \setminus A_1]$. По утверждению~\ref{pr:basic}~(iii) $H_1,H_2$ --- взвешенно-стабильные графы с меньшим числом вершин, значит существуют $T_1, T_2$, реализующие ширину $1$.  
  Рассмотрим $T=T_1 \cup T_2$. Это дерево, в котором висячие вершины --- в точности вершины графа $G$ кроме $u$, и степени всех вершин, кроме $u$ равны $3$ или $1$, а вершина $u$ имеет степень 2.
  Назовем нынешнюю вершину $u$ вершиной $v$, и повесим на нее вершину $u$ (получилось дерево $T'$). 
 Теперь все степени $3$ или $1$, и висячие вершины это вершины $G$.  Осталось проверить, что $T'$ имеет ширину $1$.
 Действительно, пусть $e$ --- ребро в $T'$, по определению $E(T') = E(T_1) \cup E(T_2) \cup \{uv\}$.  
 Ширина ребра $uv$ равна $1$, потому что размер одной из долей разбиения равен $1$. 
 Пусть $e$ является ребром $T_i$, не умаляя общности, $i=1$.
 Значит ребро $e$ разбивает вершины $G$ на два множества, одно из которых ($A = A(e)$) целиком лежит в $H_1 \setminus u$.
  Тогда 
  \[
  \cutrk_G(e)=\rk(M_G[A,G \setminus A])=\rk(M_G[A, H_1 \setminus A]),
  \]
  потому что эти матрицы отличаются на несколько нулевых строчек. 
 С другой стороны 
 \[
 \rk(M_G[A, H_1 \setminus A])=\rk(M_{H_1}[A,H_1 \setminus A])=\cutrk_{H_1}(e) = 1
 \]
 по предположению индукции.  Значит ширина $e$ равна $1$ для любого ребра $e$, то есть ширина $T'$ равна $1$.  

\textit{В двусвязном случае} граф $G$ получен копированием некоторой вершины $u$ из графа $G_1$, то есть в графе $G$ есть вершины $u,u_1$, являющиеся стягиваемой парой (иными словами $G_1=G[V \setminus u_1]$).  Поскольку $G_1$ --- взвешенно-стабильный, существует дерево $T_1$, реализующее ширину $1$. Рассмотрим его висячую вершину, соответствующую $u$, пусть ее единственный сосед $z$.
 Удалим из $T_1$ ребро $uz$ и добавим вершину $v$ вместе с ребрами $vz,vu,vu_1$; назовем полученное дерево $T$.
Теперь висячими вершинами $T$ являются в точности вершины $G$, а степени остальных вершин равны $3$. Осталось проверить, что у $T$ ширина $1$. Рассмотрим произвольное ребро из $T$; это или ребро $vu$, или $vu_1$, или ребро из $T_1$. Для первых двух случаев очевидно, что их ширина $1$, потому что одна из долей разбиения имеет размер $1$.
Пусть теперь это ребро из $T_1$. Оно разбивает множество вершин $G$ на два, одно из которых ($A = A(e)$) не содержит вершин $u,u_1$ (потому что $e$ не принадлежит пути между $u$ и $u_1$ в $T$).
Имеем
\[
\cutrk_G=\rk(M_G[A,G\setminus A])=\rk(M_G[A,G_1\setminus A]),
\]
так как столбец, соответствующий вершине $u_1$, пропорционален столбцу, соответствующему $u$ в силу стягиваемости пары $uu_1$ (то есть  нули расположены в одних и тех же строчках, а отношения ненулевых значений равны, по определению стягиваемой пары). При этом 
\[
\rk(M_G[A,G_1\setminus A])=\rk(M_{G_1}[A,G_1\setminus A])=\cutrk_{G_1}(e)=1.
\]
Значит и ширина дерева $T$ равна $1$.   

\paragraph{В другую сторону.} Покажем, что взвешенный граф ширины 1 является взвешенно дис\-тан\-ци\-он\-но-нас\-ле\-ду\-е\-мым.
Проведем доказательство по индукции по числу вершин.
\subparagraph{База.}
Если в графе не больше $3$ вершин, то он всегда взвешенный дистанционно-наследуемый.
\subparagraph{Переход.}
Пусть дан граф взвешенный $G$ и дерево $T$ со степенями $3$ и $1$ и имеющее ширину $1$.  Дерево $T$ содержит вершину $v$, на которой висят ровно две вершины $u$ и $u_1$. Утверждается, что эти вершины образуют стягиваемую пару. 
В самом деле, рассмотрим разбиение по по третьему ребру из $e = vz$. 
Части этого разбиения $A(e) = \{u, u_1\}$, $B(e) = V(G) \setminus  \{u, u_1\}$.
По определению $T$ ширина $e$ равна 1, то есть вектора $a_i = w(uv_i)$ и $b_i = w(u_1v_i)$ пропорциональны. 
Если коэффициент пропорциональности равен нулю, то одна из вершин $u,u_1$ не имеет в графе $G$ соседей вне множества $u,u_1$.
Если же коэффициент пропорциональности ненулевой, то $w(uv_i)=0$ тогда и только тогда, когда $w(u_1v_i)=0$, а для не ненулевых весов значение $\frac{w(uv_i)}{w(u_1v_i)}$ не зависит от выбора $v_i \in B(e)$.

Значит в паре $u,u_1$ либо одна из вершин висячая и висит на другой (не умаляя общности, $u_1$ висит на $u$), либо они образуют стягиваемую пару. Теперь рассмотрим граф $G \setminus u_1$ найдем для него дерево с шириной $1$. Для этого заменим в дереве $T$ вершины  $v,u,u_1$ со всеми инцидентными ребрами на одну вершину $u$ с ребром $uz$; назовем полученное дерево $T'$.
Заметим, что разбиения множества $V \setminus \{u_1\}$, задаваемые деревом $T'$ получаются из разбиений множества $V$ деревом $T$ удалением вершины $u_1$.
Значит ширина каждого разбиения $1$, и у дерева $T'$ ширина $1$. Следовательно граф $G[V \setminus u]$ имеет взвешенную ширину $1$; он взвешенный дистанционно-наследуемый по предположению индукции, a граф $G$ был получен из него копированием вершины или добавлением висячей. По теореме~\ref{mainweighted} $G$ --- взвешенный дистанционно-наследуемый граф. 

\end{proof}

\section{Заключение}

Центральная теорема~\ref{mainweighted} дает явную характеризацию взвешенных вещественно стабильных графов.
Аналогично невзвешенному случаю, из теоремы~\ref{mainweighted} и лемм~\ref{lm:copy} и~\ref{lm:gluing} следует, что степенной перечислитель взвешенного стабильного графа раскладывается на линейные множители. 

Полученный результат позволяет определить класс взвешенных дистанционно-наследуемых графов как графы. 
Оказывается, что такое определение соответствует естественному обобщению другого алгебраического определения~\cite{oum2005rank}.

Напомним, что связанные открытые вопросы можно найти в статье~\cite{cherkashin2023stability}.

\paragraph{Благодарности.} Авторы благодарны Федору Петрову за внимание к работе.
Работа Данилы Черкашина поддержана грантом Российского научного фонда номер 22-11-00131.

\bibliographystyle{plain}
\bibliography{main}

\begin{thebibliography}{1}

\bibitem{bandelt1986distance}
Hans-J{\"u}rgen Bandelt and Henry~Martyn Mulder.
\newblock Distance-hereditary graphs.
\newblock {\em Journal of Combinatorial Theory, Series B}, 41(2):182--208,
  1986.

\bibitem{brandstadt1999graph}
Andreas Brandst{\"a}dt, Van~Bang Le, and Jeremy~P. Spinrad.
\newblock {\em Graph classes: a survey}.
\newblock SIAM, 1999.

\bibitem{cherkashin2023stability}
Danila Cherkashin, Fedor Petrov, and Pavel Prozorov.
\newblock On stability of spanning tree degree enumerators.
\newblock {\em Discrete Mathematics}, 346(12):113629, 2023.

\bibitem{choe2004homogeneous}
Young-Bin Choe, James~G. Oxley, Alan~D. Sokal, and David~G. Wagner.
\newblock Homogeneous multivariate polynomials with the half-plane property.
\newblock {\em Advances in Applied Mathematics}, 32(1-2):88--187, 2004.

\bibitem{Csikvari2022ASS}
P{\'e}ter Csikv{\'a}ri and {\'A}d{\'a}m Schweitzer.
\newblock A short survey on stable polynomials, orientations and matchings.
\newblock {\em Acta Mathematica Hungarica}, 166(1):1--16, 2022.

\bibitem{oum2005rank}
Sang-il Oum.
\newblock Rank-width and vertex-minors.
\newblock {\em Journal of Combinatorial Theory, Series B}, 95(1):79--100, 2005.

\bibitem{wagner2011multivariate}
David Wagner.
\newblock Multivariate stable polynomials: theory and applications.
\newblock {\em Bulletin of the American Mathematical Society}, 48(1):53--84,
  2011.

\end{thebibliography}

\end{document}